\documentclass[UKenglish, letterpaper]{amsart}

\usepackage{amscd}
\usepackage{amssymb}
\usepackage{babel}
\usepackage{cancel}
\usepackage[T1]{fontenc}
\usepackage[utf8]{inputenc}
\usepackage{hyperref}
\usepackage[hyperpageref]{backref}
\usepackage{varioref}
\usepackage[arrow,curve,matrix]{xy}
\usepackage[retainorgcmds]{IEEEtrantools} 

\usepackage{mathpazo}
\usepackage{mathrsfs}

\usepackage{graphicx}
\usepackage{color}

\def\clap#1{\hbox to 0pt{\hss#1\hss}}

\DeclareFontFamily{OMS}{rsfs}{\skewchar\font'60}
\DeclareFontShape{OMS}{rsfs}{m}{n}{<-5>rsfs5 <5-7>rsfs7 <7->rsfs10 }{}
\DeclareSymbolFont{rsfs}{OMS}{rsfs}{m}{n}
\DeclareSymbolFontAlphabet{\scr}{rsfs}

\newcommand{\sB}{\scr{B}}

\newcommand{\sF}{\scr{F}}
\newcommand{\sG}{\scr{G}}

\newcommand{\sL}{\scr{L}}

\newcommand{\sO}{\scr{O}}

\newcommand{\bN}{\mathbb{N}}

\newcommand{\bQ}{\mathbb{Q}}
\newcommand{\bR}{\mathbb{R}}

\DeclareMathOperator{\Exc}{Exc}

\DeclareMathOperator{\codim}{codim}

\DeclareMathOperator{\Hom}{Hom}

\DeclareMathOperator{\Sym}{Sym}
\DeclareMathOperator{\supp}{supp}

\DeclareMathOperator{\Var}{Var}

\DeclareMathOperator{\vol}{vol}

\newcommand{\wtilde}{\widetilde}

\newcounter{thisthm}

\newcommand{\iref}[1]{(\thesection.\the\value{thisthm}.\the\value{#1})}

\theoremstyle{plain}    
\newtheorem{thm}{Theorem}[section]

\numberwithin{equation}{thm}
\numberwithin{figure}{section}
\theoremstyle{plain}    
\newtheorem{cor}[thm]{Corollary}

\newtheorem{conjecture}[thm]{Conjecture}

\theoremstyle{plain}    
\newtheorem{prop}[thm]{Proposition}
\newtheorem{proclaim-special}[thm]{\specialthmname}

\theoremstyle{remark}

\newtheorem{defn}[thm]{Definition}
\newtheorem{rem}[thm]{Remark}
\newtheorem{explanation}[thm]{Explanation}

\newtheorem{subclaim}[equation]{Claim} 
\newtheorem*{claim*}{Claim} 
\newtheorem{notation}[thm]{Notation}

\newtheoremstyle{bozont-remark}{3pt}{3pt}%
     {}
     {}
     {\it}
     {.}
     {.5em}
     {\thmname{#1}\thmnumber{ #2}: \thmnote{\sc #3}}
\theoremstyle{bozont-remark}

\def\factor#1.#2.{\left. \raise 2pt\hbox{$#1$} \right/\hskip -2pt\raise
  -2pt\hbox{$#2$}}

\newlength{\swidth}
\newenvironment{enumerate-p}{
  \begin{enumerate}}
  {\setcounter{equation}{\value{enumi}}\end{enumerate}}

\definecolor{tomato}{RGB}{180,62,39}
\definecolor{forrest}{RGB}{81,133,49}
\definecolor{lighttomato}{RGB}{253,65,65}
\definecolor{lightforrest}{RGB}{145,237,87}
\definecolor{mygreen}{RGB}{40,104,69}
\definecolor{mygreen2}{RGB}{3,149,39}
\definecolor{darkolivegreen}{RGB}{102,118,75}
\definecolor{cranegreen}{RGB}{102,118,75}
\definecolor{mydarkblue}{RGB}{10,92,153}
\definecolor{myblue}{RGB}{57,222,186}
\definecolor{pinkish}{RGB}{213,83,222}
\definecolor{colD}{RGB}{213,83,222}
\definecolor{defb}{RGB}{213,83,222}
\definecolor{goldenrod}{RGB}{225,115,69}
\definecolor{mauve}{RGB}{224, 176, 255}
\definecolor{fuchsia}{RGB}{255, 0, 255}
\definecolor{lavender}{RGB}{230, 230, 250}
\definecolor{gold}{RGB}{255, 215, 0}
\definecolor{orange}{RGB}{255, 127, 0}
\definecolor{maroon}{RGB}{123, 17, 19}
\definecolor{brightmaroon}{RGB}{195, 33, 72}
\definecolor{richmaroon}{RGB}{176, 48, 96}
\definecolor{green}{RGB}{3,149,39}

\author{Behrouz Taji}

\title[Isotriviality of families of canonically-polarized manifolds]{The isotriviality of smooth families of canonically-polarized manifolds over a special quasi-projective base}

\address{The Department of Mathematics and Statistics, McGill University. Burnside Hall, Room 1031. 805 Sherbrooke W. Montreal, QC, H3A 0B9, Canada}

\email{\href{mailto:behrouz.taji@matil.mcgill.ca}{behrouz.taji@mail.mcgill.ca}}

\date{\today}

\hypersetup{
    pdfstartview={Fit},
  pdfpagelayout={TwoColumnRight},
  pdfpagemode={UseOutlines},
  bookmarks,
  colorlinks}

\begin{document}

\begin{abstract} In this paper we prove that a smooth family of canonically polarized manifolds parametrized by a special (in the sense of Campana) quasi-projective variety is isotrivial.
\end{abstract}

\maketitle

\section{introduction}

In 1962 Shafarevich conjectured that any smooth family of curves of genus $g\geq 2$ over non-hyperbolic algebraic curves, namely $\mathbb{C}$, $\mathbb{C}^*$, $\mathbb{P}^1$ and elliptic curve $E$, is isotrivial. More generally, it was conjectured that any smooth family of \emph{canonically-polarized manifolds} over these curves has no (algebraic) variation. This gives rise to a natural question: Is there a large class of higher dimensional bases over which every such family is isotrivial? Campana has introduced \emph{special} varieties as higher dimensional analogues of non-quasi-hyperbolic curves (the curves listed above) and conjectured that they serve as natural candidates for such bases.

\begin{conjecture}[The Isotriviality Conjecture of Campana] \label{iso} Let $Y^{\circ}$ be a smooth quasi-projective variety parametrizing a smooth family of canonically polarized manifolds. If $Y^{\circ}$ is special (see the definition below), then the family is isotrivial.

\end{conjecture}

\begin{defn}[Special Logarithmic Pairs]\label{special} Let $(Y,D)$ be a pair consisting of a smooth projective variety $Y$ and a simple normal-crossing reduced boundary divisor $D$. We call $(Y,D)$ special, if for every invertible subsheaf  $\sL \subseteq \Omega_Y^p \log(D)$ and $p>0$, we have $\kappa(\sL)<p$. Moreover we shall call a smooth quasi-projective variety $Y^{\circ}$ special, if (Y,D) is special as a logarithmic pair, where $Y$ is a smooth compactification with a simple normal-crossing (snc, for short) boundary divisor $D$.

\end{defn}

So, by definition $\mathbb{C}$, $\mathbb{C}^*$, $\mathbb{P}^1$ and $E$ is the list of all special quasi-projective curves. Other important examples of special varieties include rationally-connected varieties, varieties with zero Kodaira dimension~\cite[Thm.~5.1]{Ca04} and those with nef anti-canonical divisor~\cite[Thm.~11.1]{Lu02}. These examples however are very particular instances of special varieties and it is important to recall that in every
dimension $n$, there are special quasi-projective manifolds of all possible
Log-Kodaira dimensions $<n$.

Conjecture~\ref{iso} is generalization of the following celebrated conjecture of Viehweg.

\begin{conjecture}[Viehweg's Hyperbolicity Conjecture]\label{VC} Let $f^\circ: X^\circ \to Y^\circ$ be a smooth family of canonically-polarized varieties over a quasi-projective variety $Y^\circ$. Assume that Y is a smooth compactification of $Y^\circ$ with snc boundary divisor $D\cong Y\backslash Y^\circ$. If $\Var(f^\circ)$ is maximal, then $(Y,D)$ is of log-general type. (see ~\cite[Introduction]{Vie83} to for the definition of $\Var(f^\circ)$).
 
\end{conjecture}

This latter conjecture (Conjecture~\ref{VC}) has been recently established in~\cite{CP13} by using, among many other things, an important generalization of Miyaoka's generic semi-positivity (see Theorem~\ref{ogsp}) and the following remarkable result of Viehweg and Zuo.

\begin{thm}[\protect{Existence of Pluri-Logarithmic Forms in the Base, cf.~\cite[Thm.~1.4]{VZ02}}]\label{p-sections} The notations are the same as in Conjecture~\ref{VC}. If $f^\circ$ is \emph{not} isotrivial, then for a positive integer $N\in \mathbb{N}^+$, there exists an invertible subsheaf $\sL\subseteq \Sym^N\bigl(\Omega_Y\log (D)\bigr)$ such that $\kappa(\sL)\geq \Var(f^\circ)$.
\end{thm}
Clearly, Viehweg's hyperbolicity conjecture~\ref{VC} in the case of $\dim Y^\circ=1$ is an immediate corollary of Theorem~\ref{p-sections}. Interestingly much more is true; in~\cite{VZ03} Viehweg and Zuo  prove the Brody-hyperbolicity of the moduli stack $\mathcal{M}$: the quasi-projective variety $Y^\circ$ serving as the base of a family $(f^\circ: X^\circ \to Y^\circ)\in \mathcal{M}(Y^\circ)$ for which the induced moduli map $\mu:Y^\circ \to \mathfrak{M}$ is generically finite, is Brody-hyperbolic. Here $\mathfrak{M}$ is the quasi-projective scheme (\cite{Vie95}) equipped with transformations $$\Psi: \mathcal{M}\to \Hom(.,\mathfrak{M}),$$ such that $\mathfrak{M}$ is the coarse moduli scheme of the moduli functor $\mathcal{M}$ of smooth family of canonically-polarized manifolds.

\

Viehweg's conjecture~\ref{VC} were already known in $\dim(Y^\circ) \leq 3$ by Kebekus and Kov\'acs~\cite[Thm.~1.1]{KK08}. The stronger conjecture of Campana (Conjecture~\ref{iso}) has also been established when $\dim(Y^\circ)\leq 3$, thanks to Kebekus and Jabbusch~\cite[Thm.~1.5]{JK11b}. In the final section (Section~\ref{last}), and after following Campana and P\u{a}un's proof of Viehweg's conjecture very closely, we give a proof to the isotriviality conjecture~\ref{iso}. The proof heavily depends on a recent generic semi-positivity result of Campana and P\u{a}un, existence of log-minimal models for klt pairs with big boundary divisors established by~\cite[Thm.~1.1]{BCHM10}, and an important refinement of Theorem~\ref{p-sections} given by~\cite[Thm.~1.4]{JK11a}.

\

\begin{thm}[Isotriviality of Smooth Families of Canonically-Polarized Manifolds]\label{ison} The isotriviality conjecture~\ref{iso} holds in all dimensions.

\end{thm}

\

According to Campana's reduction theory, for every projective variety $Y$ there exists an almost holomorphic map $C_Y:Y\dashrightarrow Z$, called the \emph{core}, whose general fiber is special and contracts almost all special subvarieties of $Y$. As a result of Theorem~\ref{ison} it follows that the moduli maps associated to smooth families of canonically-polarized manifolds factors through the core.

\begin{cor}[Factorization of the Moduli Map Through the Core]\label{core} Let $Y^\circ$ be a smooth quasi-projective variety admitting a morphism $\mu:Y^\circ \to \mathfrak{M}$, where $\mu=\Psi \bigl(\mathcal{M}(Y^\circ)\bigr)$. Let $\wtilde \mu$ be the induced morphism between smooth compactifications $Y$, $\overline{\mathfrak{M}}$ of $Y$ and $\mathfrak{M}$, respectively. Then, $\wtilde \mu$ factors through the core $C_{(Y,D)}:(Y,D)\dashrightarrow Z$ associated to a smooth 
compactification $(Y,D)$ of $Y^\circ$.

\end{cor}

Notice that Corollary~\ref{core} immediately implies that Viehweg's hyperbolicity conjecture (already settled in~\cite{CP13}) holds: Let $f^\circ:X^\circ \to Y^\circ$ and $(Y,D)$ be as in Theorem~\ref{VC}. If $(Y,D)$ is not of log-general type, then $C_Y:Y \dashrightarrow Z$ has positive-dimensional general fibres. On the other hand, by Corollary~\ref{core}, the moduli map $\mu:Y\to \overline{\mathfrak{M}}$ factors through $C_Y$. But by the assumption $\mu$ is generically-finite, a contradiction. 

The proof of Theorem~\ref{ison} essentially consists of the following two streps: First we use Viehweg-Zuo's factorization result (Theorem~\ref{p-sections}), together with its refinement by~\cite{JK11a}, to reduce the problem to the following (see Theorem~\ref{reduction} for details): Given a smooth pair $(X,D)$, existence of an invertible subsheaf $\sL\subseteq \bigl(\Omega_X\log(D)\bigr)^{\otimes_{\mathcal{C}}N}$ (see Definition~\ref{tensorial}), for some $N\in\mathbb{N}^+$, with maximal $\mathcal{C}$-Kodaira dimension (this is defined in Definition~\ref{okodaira}) implies that $(X,D)$ is of log-general type. The second step (Section~\ref{last}) is to prove this statement using the positivity result of~\cite{CP13} and results of~\cite{BCHM10} (Theorem~\ref{final}).

\subsection{Acknowledgements}

The author would like to thank his advisor, S. Lu, for his invaluable help, guidance and support. He also owes a debt of gratitude to F. Campana for his encouragements, generosity, and many fruitful discussions. The author would like to express his sincere thanks to S. Kebekus for his careful reading of the first draft of this paper and many kind and inspiring suggestions. A special thanks is owed to E. Rousseau, P. Cascini, B. Claudon, M. Roth and J. Hurtubise for their interest and helpful comments. 

\

\section{Preliminaries}

To approach the isotriviality conjecture~\ref{iso}, it is essential to work with pairs (or the orbifold pairs in the sense of Campana) instead of just logarithmic ones. We refer the reader to~\cite{Ca08} and~\cite{JK11b} for an in-depth discussion of the definitions and background. In the present section we give a brief overview of the key ingredients of this theory to the extent that is necessary for our arguments in the rest of the paper.

\begin{defn}[Smooth Pairs]\label{orbipairs} Let $X$ be an $n$-dimensional normal (quasi-) projective variety and $D=\sum d_iD_i$, where $d_i\in \mathbb{Q}\cap [0,1]$, a $\mathbb{Q}$-Weil divisor in $X$. We shall call the pair $(X,D)$ a \emph{smooth pair},  if $X$ is smooth and $\supp(D)$ is simple normal-crossing.
\end{defn}

\

\begin{defn}[$\mathcal{C}$-Multiplicity]\label{multiplicity} Let $(X,D)$ be a smooth pair as in Definition~\ref{orbipairs}. When $d_i\neq 1$, let $a_i$ and $b_i$ be the positive integers for which the equality $1-\frac{b_i}{a_i}=d_i$ holds. For every $i$, we define the $\mathcal{C}$-multiplicity of the irreducible component $D_i$ of $D$ by 

$$
  m_D(D_i) := \left\{ 
    \begin{matrix}
      \frac{1}{1-d_i}=\frac{a_i}{b_i}& \text{if $d_i\neq 1$ } \\
      \infty & \text{ if $d_i=1.$  }
    \end{matrix}
  \right.
  $$

\end{defn}

\

A classical result of Kawamata (see~\cite[Prop.~4.1.12]{Laz04}) proves that given a collection of smooth prime divisors $\{D_1,\ldots, D_l\}$ and positive integers $\{c_1,\ldots,c_l\}$, one can always construct a \emph{smooth} variety $Y$ together with a finite, flat morphism $\gamma:Y\to X$ such that
\begin{equation*}
\gamma^*(D_i)=c_i\sum D_{ij},
\end{equation*}
where $(\sum D_{ij})$ is a simple normal-crossing divisor in $Y$. In particular, given a smooth pair $(X,D)$, we may take the coefficients $c_i$ to be equal to $a_i$ ($a_i$ being the numerator of $m_D(D_i)$, as in Definition~\ref{multiplicity}), so that the resulting Kawamata cover $\gamma:Y\to X$ is, in a sense, \emph{adapted} to the structure of the pair $(X,D)$.

\begin{defn}[Adapted Covers]\label{covers} Let $(X,D)$ be a smooth pair, $Y$ a smooth variety, and $\gamma:Y \to X$ a finite, flat, Galois cover with Galois group $G$ such that if $m_D(D_i)=\frac{a_i}{b_i}<\infty$, then every prime divisor in $Y$ that appears in $\gamma^*(D_i)$ has multiplicity exactly equal to $a_i$. We call $\gamma$ an adapted cover for the pair $(X,D)$, if it additionally satisfies the following properties:

\begin{enumerate}
\item The branch locus is given by $$\supp(H+\bigcup_{m_D(D_i)\neq \infty} D_i),$$ where H is a general member of a linear system $|L|$ of a very ample divisor $L$ in $X$.
\item $\gamma$ is totally branched over $H$.
\item $\gamma$ is not branched at the general point of $\supp(\lfloor D \rfloor)$.

\end{enumerate}

\end{defn}

\

\begin{notation} Let $\gamma:Y\to X$ be an adapted cover of a smooth pair $(X,D)$, where $D=\sum d_iD_i$, $d_i=1-\frac{b_i}{a_i}$ as in Definition~\ref{multiplicity}. For every prime component $D_i$ of $D$ with $m_D(D_i)\neq \infty$, let $\{D_{ij}\}_{j(i)}$ be the collection of prime divisors that appear in $\gamma^{-1}(D_i)$. We define new divisors in Y by 
\begin{flalign}\label{not}
&D_Y^{i,j}:=b_iD_{ij} \ , \quad m_D(D_i)\neq \infty,\\
&D_{\gamma}:=\gamma^*(\lfloor D\rfloor).
\end{flalign}
\end{notation}

\

\begin{defn}[$\mathcal{C}$-Cotangent Sheaf]\label{ocs} Given a smooth pair $(X,D)$ with an adapted cover $\gamma:Y\to X$, define the $\mathcal{C}$-cotangent sheaf $\Omega_{Y^{\partial}}$ to be the unique maximal locally-free subsheaf of $\Omega_Y\log(D_{\gamma})$ for which the sequence 

$$
\xymatrix{
0 \ar[r] & \Omega_{Y^{\partial}}|_{(Y\backslash D_{\gamma})} \ar[r]  & \gamma^*\bigl(\Omega_X\log (\ulcorner D \urcorner)\bigr)|_{(Y\backslash D_{\gamma})} \ar[r]^(.7){\rho}  & \bigoplus \limits_{i,j(i)}\sO_{D_Y^{i,j}} \ar[r] & 0,
}
$$
induced by the natural residue map, is exact. 

\end{defn}

\begin{rem} The $\mathcal{C}$-cotangent sheaf defined in~\ref{ocs} coincides with Campana and P\u{a}un's notion~\cite[Sec.~1.1]{CP13} of the coherent sheaf on $Y$ which they denote by $\gamma^*\Omega^1(X,D)$. It is also identical with the sheaf defined in~\cite[Lem.~4.2]{Lu02}. See also~\cite[Def.~2.13]{JK11b} for an equivalent definition in the classical setting, i.e. when the $\mathcal{C}$-multiplicities are all integral. 

\end{rem}

\begin{notation} We shall denote the dual of the $\mathcal{C}$-cotangent sheaf by $T_{Y^{\partial}}$, i.e. $$T_{Y^{\partial}}:=(\Omega_{Y^{\partial}})^*$$.
\end{notation}

\begin{rem}[Determinant of $\mathcal{C}$-Cotangent Sheaf]\label{det-rem} Given a smooth pair $(X,D)$, let $\gamma:Y \to X$ be an adapted cover of degree $d$. There exists a natural isomorphism between the two invertible sheaves $\det(\Omega_{Y^{\partial}})$ and $\sO_Y\bigl(\gamma^*(K_X+D)\bigr)$
\begin{equation}\label{det}
     \det (\Omega_{Y^{\partial}})\cong \sO_Y\bigl(\gamma^*(K_X+D)\bigr).
         \end{equation}

\noindent This follows from the ramification formula for the adapted cover $\gamma$: 

\begin{align*}
K_Y+D_{\gamma}  &=\gamma^*(K_X+\lfloor D \rfloor)+\sum_{\substack{i \\ m_D(D_i)\neq \infty}} \sum_{j(i)}(a_i-1)D_{ij}+(d-1)\wtilde H\\
     &=\gamma^*(K_X+D)-\gamma^*(D-\lfloor D \rfloor)+\sum_{\substack{i \\ m_D(D_i)\neq \infty}} \sum_{j(i)}(a_i-1)D_{ij}+(d-1)\wtilde H\\
     &=\gamma^*(K_X+D)-\sum_{\substack{i  \\ m_D(D_i)\neq \infty}} \sum_{j(i)}(a_i-b_i)D_{ij}+ \sum_{i,m_D(D_i)\neq \infty}\sum_{j(i)}(a_i-1)D_{ij}\\
     & \qquad \qquad  \qquad \qquad \qquad \qquad \qquad \qquad \qquad \qquad \qquad \qquad \quad+(d-1)\wtilde H\\ 
     &=\gamma^*(K_X+D)+\sum_{\substack{i \\ m_D(D_i)\neq \infty}} \sum_{j(i)}(b_i-1)D_{ij}+(d-1)\wtilde H,\\
   \end{align*}

\noindent for $\wtilde H:=\gamma^*H$, where $H$ is the very ample divisor given in Definition~\ref{covers}. As a consequence, we find that the isomorphism (\ref{det}) holds by construction:

\begin{align*}
\det \Omega_{X^{\partial}}  &\cong \sO_Y\bigl((K_Y+D_{\gamma})-\sum_{\substack{i \\ m_D(D_i)\neq \infty}}\sum_{j(i)}(b_i-1)D_{ij}- (d-1)\wtilde H\bigr) && \text{by definition}\\      
        &\cong \sO_Y\bigl(\gamma^*(K_X+D)\bigr), \\
    \end{align*}
where the last isomorphism follows form the ramification formula. Clearly, the isomorphism (\ref{det}) implies that the $\mathcal{C}$-cotangent sheaf $\Omega_{Y^{\partial}}$ can be seen as the unique locally-free subsheaf of $\Omega_{Y}\log(D_{\gamma})$ whose determinant is isomorphic to the pull-back bundle $\sO_Y\bigl(\gamma^*(K_X+D)\bigr)$.

\end{rem}

\

\begin{defn}[Symmetric $\mathcal{C}$-Differential Forms, cf. \protect{\cite[Sect.~2.6-7]{Ca08}}]\label{sodf} Let $(X,D)$ be a smooth pair, $D=\sum d_iD_i$, and $V_x$ an open neighbourhood of a given point $x\in X$ equipped with a coordinate system $z_1,\ldots,z_n$ such that $\supp(D)\cap V_x=\{z_1\cdot \ldots\cdot z_l=0\}$, for a positive integer $1\leq l\leq n$. For every $N\in\mathbb{N}^+$, define the sheaf of symmetric $\mathcal{C}$-differential forms $\Sym_{\mathcal{C}}^N\bigl(\Omega_X\log (D)\bigr)$ by the locally-free subsheaf of $\Sym^N \bigl(\Omega_X\log(\ulcorner D \urcorner)\bigr)$ that is locally-generated, as an $\sO_{V_x}$-module, by the elements 
 $$\frac{dz_1^{k_1}}{z_1^{\lfloor d_1\cdot k_1\rfloor}} \cdot \ldots \cdot \frac{dz_l^{k_l}}{z_l^{\lfloor d_l\cdot k_l \rfloor}} \cdot dz_{l+1}^{k_{l+1}} \cdot \ldots \cdot dz_n^{k_n},$$

\noindent where $\sum k_i =N$.

\end{defn}

\

\begin{rem}[An Equivalent Definition]\label{equiv} There is an alternative definition for the sheaf of $\mathcal{C}$-differential forms: Let $V_x$ be an open neighbourhood of $x\in X$ as in Definition~\ref{sodf} and take $\gamma:W \to V_x$ to be an adapted cover for $(V_x,D|_{V_x})$. Let $\sigma \in \Gamma\bigl(V_x,\Sym^{N}\bigl(\Omega_X(*\ulcorner D \urcorner)\bigr)\bigr)$, that is $\sigma$ is a local rational section of $\Sym^{N}(\Omega_X)$ with poles along $\ulcorner D \urcorner$. Then,
 
\begin{equation}\label{alternative}
   \sigma \in \Gamma\bigl(V_x,\Sym_{\mathcal{C}}^N\bigl(\Omega_X\log (D)\bigr)\bigr) \iff  \gamma^{*}(\sigma)\in    \Gamma\bigl(W,\Sym^{N}(\Omega_{W^{\partial}})\bigr),
        \end{equation}

\

 \noindent So that, in particular, $\gamma^*(\sigma)$ has at worst logarithmic poles \emph{only} along those prime divisors in $W$ that dominate $(\lfloor D \rfloor\cap V_x)$, and is regular otherwise.
 
 \begin{explanation} Assume that $\sigma\in \Gamma\bigl(V_x,\Sym_{\mathcal{C}}^N\bigl(\Omega_X\log (D)\bigr)\bigr)$ is a local $\mathcal{C}$-differential form in the sense of (\ref{alternative}). By the classical result of Iitaka~\cite[Chap.~11]{Iit82}, it follows that $\sigma \in \Gamma\bigl(V_x,\Sym^N\bigl(\Omega_X\log(\ulcorner D \urcorner)\bigr)\bigr)$. In particular we find that along the reduced component of $D$ the equivalence between the two definitions trivially holds. So assume, without loss of generality, that $m_D(D_i)\neq \infty$, for all irreducible components $D_i$ of $D$. Furthermore let us assume, for simplicity, that 
  
$$\sigma=f\cdot \frac{dz_1^{k_1}}{z_1^{e_1}}\cdot \ldots \cdot \frac{dz_l^{k_l}}{z_l^{e_l}}\cdot dz_{l+1}^{k_l+1}\cdot \ldots \cdot dz_n^{k_n} \ \ \in \Gamma\bigl((V_x,\Sym^{N}\bigl(\Omega_X\log(\ulcorner D \urcorner)\bigr)\bigr) ,$$ where $f\in \sO_{V(x)}$ with no zeros along $D_i$'s, is the local explicit description of $\sigma$. Since $\gamma^*(\sigma)\in \Sym^N(\Omega_{W^{\partial}})$, the inequality 
$$k_i\cdot(a_i-1)-a_i\cdot e_i \geq k_i(b_i-1)$$ holds for $1\leq i\leq l$, where $d_i=1-(b_i/a_i)$, i.e.

$$e_i \leq k_i.d_i , \text{\ \ for all}\  1\leq i\leq l .$$ In particular $\sigma$ is a symmetric $\mathcal{C}$-differential form on $V_x$ in the sense of Definition~\ref{sodf}.

\end{explanation}

\end{rem}

\begin{rem}[Tensorial $\mathcal{C}$-Differential Forms]\label{tensorial} Similar to the Definitions~\ref{sodf} and (\ref{alternative}), we can define the sheaf of tensorial $\mathcal{C}$-differential forms $\bigl(\Omega_X\log (D)\bigr)^{\otimes_{\mathcal{C}}N}$ as the maximal subsheaf of $\bigl(\Omega_X\log(\ulcorner D \urcorner)\bigr)^{\otimes N}$ such that 
\begin{equation*}
\gamma^*\Bigl(\bigl(\Omega_X\log (D)\bigr)^{\otimes_{\mathcal{C}}N}\Bigr)\subseteq (\Omega_{Y^{\partial}})^{\otimes N}.
\end{equation*}
Using the notations in Remark~\ref{equiv} , pluri-$\mathcal{C}$-differential forms are locally defined as follows: 

\begin{equation}\label{alternative-tensor}
   \sigma \in \Gamma\bigl(V_x,\bigl(\Omega_X\log (D)\bigr)^{\otimes_{\mathcal{C}}N}\bigr) \iff  \gamma^{*}(\sigma)\in    \Gamma\bigl(W,(\Omega^{\otimes N}_{W^{\partial}})\bigr),
        \end{equation}

\end{rem}

As we shall see in section~\ref{VZJK}, the Viehweg-Zuo subsheaves generically come from the coarse moduli space, as long as we extend the sheaf of symmetric differential forms to that of $\mathcal{C}$-differential forms associated to the naturally imposed $\mathcal{C}$-structures or orbifold structures (see Definition~\ref{C-base} below or~\cite[Sect.~3]{Ca08}) that appear over the moduli variety. But, as the usual Kodaira dimension of subsheaves of symmetric $\mathcal{C}$-differential forms is not sensitive to the fractional positivity of the non-reduced components of the boundary divisor (see Remark~\ref{compare} below), a new birational notion is needed to measure the positivity of the Viehweg-Zuo subsheaves in the moduli.

\begin{defn}[$\mathcal{C}$-Kodaira Dimension, cf. \protect{\cite[Sect.~2.7]{Ca08}}]\label{okodaira} Let $(X,D)$ be a smooth pair and $\sL\subseteq \bigl(\Omega_X\log (D)\bigr)^{\otimes_{\mathcal{C}}r}$ a saturated coherent subsheaf of rank one. Define the $\mathcal{C}$-product $\sL^{\otimes_{\mathcal{C}}m}$ of $\sL$, to the order of $m$, to be the saturation of the image of  $\sL^{\otimes m}$ inside $\bigl(\Omega_X\log (D)\bigr)^{\otimes_{\mathcal{C}}(m.r)}$ and define the $\mathcal{C}$-Kodaira dimension of $\sL$ by $$\kappa_{\mathcal{C}}(X,\sL):=\max \{k \; | \; \limsup_{m \to \infty} \frac{h^0\bigl(X,\sL^{\otimes_{\mathcal{C}}m}\bigr)}{m^k} \neq 0 \}, $$ and when $h^0\bigl(X,\sL^{\otimes_{\mathcal{C}}m}\bigr)=0$ for all $m\in \mathbb{N^+}$, then, by convention, we define $\kappa_{\mathcal{C}}(X,\sL)=-\infty$.

\end{defn}

\begin{rem} [Comparing Kodaira Dimensions]\label{compare} When $D=0$ or when $D$ is reduced the sheaf of pluri-$\mathcal{C}$-differential forms $\bigl(\Omega_X\log (D)\bigr)^{\otimes_{\mathcal{C}}r}$ is equal to $(\Omega_X)^{\otimes r}$ and $\bigl(\Omega_X\log (D)\bigr)^{\otimes r}$, respectively, so that the $\mathcal{C}$-Kodaira dimension $\kappa_{\mathcal{C}}(X,\sL)$ of a rank one coherent subsheaf $\sL$ of $\bigl(\Omega_X\log (D)\bigr)^{\otimes_{\mathcal{C}}r}$ coincides with the usual Kodaira dimension $\kappa(X,\sL)$ of $\sL$.

\end{rem}

\

Let $(Y,D)$ be a smooth pair, $Z$ a smooth variety, and $f:Y\to Z$ a fibration with connected fibres. Assume that every $f$-exceptional prime divisor $F$, that is, $\codim_Z\bigl(f(F)\bigr)\geq 2$, is a reduced component of $D$. Then, simple local calculations show that there exists a maximal---in the sense of multiplicities of the irreducible components---divisorial structure $\Delta$ on $Z$, whose support coincides with the codimension-$1$ closed subset of the log-discriminant locus $B$ of $f:(Y,D)\to Z$ and that the natural pull-back map 
\begin{equation*}
(df)^m: f^*\bigl(\Sym_{\mathcal{C}}^m\bigl(\Omega_Z\log(\Delta)\bigr)\bigr) \to \Sym_{\mathcal{C}}^m\bigl(\Omega_Y\log(D)\bigr)
\end{equation*}
is well-defined. Recall that the log-discriminant locus $B$ is the smallest closed subset of $Z$ such that $f$ is smooth over its complement, and that for every point $z\in Z\backslash \Delta$, the set-theoretic fibre $f^{-1}(z)$ is not contained in $D$, and that the scheme-theoretic intersection of the fibre $Y_{z}$ with $D$ is a simple normal-crossing divisor in $Y_{z}$. We call $\Delta$ the \emph{$\mathcal{C}$-base} (or the orbifold-base) of the fibration $f:(Y,D)\to Z$.

\begin{defn}[$\mathcal{C}$-Base of a Fibration]\label{C-base} Given a smooth pair $(Y,D)$, let $f:Y\to Z$ be a fibration with connected fibres onto a smooth variety $Z$. Let $\{\Delta_i\}_i$ be the set of the irreducible components of the divisorial part of the log-discriminant locus of $f$. For every $i$, define $\{\Delta_{ij}\}_j$ to be the collection of prime divisors in $f^{-1}(\Delta_i)$ that are \emph{not} $f$-exceptional. To each divisor $\Delta_i$, assign a positive rational number $m_{\Delta}(\Delta_i)$ defined by 
\begin{equation*}
m_{\Delta}(\Delta_i):=\min_j\{d_j\cdot m_{\Delta}(\Delta_{ij})\},
\end{equation*}
$d_j$ being the positive integer verifying the equality 
\begin{equation*}
f^*(\Delta_i)=\sum_j d_j\Delta_{ij}+E.
\end{equation*}
We define the $\mathcal{C}$-base of the fibration $f:(Y,D)\to Z$ by the divisor 
\begin{equation*}
\Delta:=\sum_i\bigl(1-\frac{1}{m_{\Delta}(\Delta_i)}\bigr)\Delta_i.
\end{equation*}

\end{defn}

\

We finish this section by collecting the various notations that we have introduced in the following table.

\begin{table}[htb!]
\caption{Notations}
\centering
\begin{tabular}{c      c}
\hline\hline
$\Omega_{Y^{\partial}}$  &  $\mathcal{C}$-cotangent sheaf~\ref{ocs}\\
$\Sym_{\mathcal{C}}^N\bigl(\Omega_X\log (D)\bigr)$   &  Symmetric $\mathcal{C}$-differential forms~\ref{sodf}\\
$\bigl(\Omega_X\log (D)\bigr)^{\otimes_{\mathcal{C}}N}$   & Tensorial $\mathcal{C}$-differential forms~\ref{tensorial}\\
$\sL^{\otimes_{\mathcal{C}}N}$                               &  $\mathcal{C}$-product~\ref{okodaira} \\
$\kappa_{\mathcal{C}}(X,\sL)$  & $\mathcal{C}$-Kodaira dimension~\ref{okodaira}\\
\hline
\end{tabular}
\label{notat}
\end{table}

\

\section{The orbifold generic semi-positivity}

Miyaoka's generic semi-positivity result (see~\cite{Miy87} and~\cite{Miy85}) establishes a correspondence between abundance of rational curves (uniruledness) on a smooth projective variety and generic (semi-)positivity of the cotangent sheaf $\Omega_X$. On the other hand the results of~\cite{BDPP} prove that uniruledness of $X$ is characterized by the pseudo-effectivity of $K_X$. This suggests---as was originally formulated by Camapna---that a generalization of Miyaoka's result in the logarithmic context should read as follows: Pseudo-effectivity of $K_X+D$ implies the generic semi-positivity of $\Omega_X\log(D)$. But the positivity result of Miyaoka was achieved via certain characteristic $p$-arguments which cannot be adapted to the context of pairs. Nevertheless, in~\cite{CP13}, Campana and P\u{a}un overcome this obstacle by using Bogomolov-McQuillan criterion for the
algebraicity of foliations induced by positive subsheaves of the tangent sheaf.

\begin{thm}[Generic Semi-Positivity of $\mathcal{C}$-Cotangent Sheaf \protect{\cite[Thm.~2.1]{CP13}}]\label{ogsp} Let $(X,D)$ be a smooth pair with an adapted cover $\gamma:Y\to X$, whose Galois group we denote by $G$. If $(K_X+D)$ is pseudo-effective, then every torsion-free, coherent, $\sO_Y$-module quotient $\sF$ of $(\Omega_{Y^{\partial}})^{\otimes N}$ verifies the inequality
\begin{equation}
\mathrm{c}_1(\sF)\cdot \gamma^*(H_1)\cdot \ldots\cdot \gamma^*(H_{n-1})\geq 0,
\end{equation}
for all $(n-1)$-tuples of ample divisors $(H_1,\ldots, H_{n-1})$ in $X$.

\end{thm}

\
As an immediate corollary we get an inequality involving the intersection of $(K_X+D)$ and any invertible subsheaf $\sL\subseteq \bigl(\Omega^1_X\log (D)\bigr)^{\otimes_{\mathcal{C}}N}$ with nef divisors:

\begin{cor}\label{invariance} Let $(X,D)$ be a smooth pair of dimension $n$. Let $\sL\subseteq \bigl(\Omega^1_X\log (D)\bigr)^{\otimes_{\mathcal{C}}N}$ be an invertible subseheaf and $L$ a divisor in $X$ such that  $\sO_X(L)\cong \sL$. If $(K_X+D)$ is pseudo-effective, then for every collection of $(n-1)$ $\mathbb{Q}$-Cartier nef divisors $P_1,\ldots, P_{n-1}$ the following inequality holds:
$$(N.(n^N)^{N-1}.(K_X+D)-L)\cdot P_1 \cdot \ldots \cdot P_{n-1}\geq 0.$$

\end{cor}

\

\section{Viehweg-Zuo subsheaves in the parametrizing space}\label{VZJK}

The result of Jabbusch and Kebekus~\cite{JK11b} shows that the $\mathcal{C}$-differential forms is the correct framework to study the positivity of subsheaves of forms in the coarse moduli space of canonically-polarized manifolds. In this section we give a brief explanation of how one can then reduce the isotriviality conjecture (Conjecture~\ref{iso}) to the problem of showing that existence of rank one subsheaves of the sheaf pluri-$\mathcal{C}$-differential forms, attached to a smooth pair, with maximal $\mathcal{C}$-Kodaira dimension implies that the given pair is of log-general type (see Theroem~\ref{reduction} below). To prepare the correct setting for this reduction, we introduce a notion that, as far as the author is aware, is originally due to Campana (see~\cite[Sect~1.1]{Ca08}).

\begin{defn}[Neat Model of a Pair]\label{neat} Let $(Y,D)$ be a normal logarithmic pair ($Y$ is normal and the Weil divisor $D$ is reduced) and $h:Y\to Z$ a fibration with connected fibers onto an algebraic base $Z$. We call a smooth pair $(Y_h,D_h)$ a neat model for $(Y,D)$ and h, if there exists a fibration $\wtilde h:Y_h \to Z_h$ that is birationally equivalent to $h$, that is, there are birational morphisms $\mu:Y_h \to Y$ and $\alpha:Z_h \to Z$ such that the diagram

 $$
  \xymatrix{
     Y \ar[d]_{h}  &Y_h  \ar[l]_{\mu} \ar[d]^{\wtilde h}  \\
     Z   & Z_h \ar[l]_{\alpha}
  }
  $$
\noindent commutes, for which the following conditions are satisfied:

\begin{enumerate}
\item\label{neat2} $D_h$ is the extension of the $\mu$-birational transform $\wtilde D$ of $D$ by some reduced $\mu$-exceptional divisor, i.e. $D_h=\wtilde D+E'$, where $E'$ is $\mu$-exceptional.
\item\label{neat1} $(Z_h,\Delta_h)$ is a smooth pair, $\Delta_h$ being the $\mathcal{C}$-base (see Definition~\ref{C-base}) of the fibration $\wtilde h:(Y_h,D_h)\to Z_h$.
\item\label{neat3} Every $\wtilde h$-exceptional prime divisor $P$ in $Y_h$ ($P$ verifies the inequality $\codim_{Z_h}(\wtilde h(P))\geq 2$) is contained in $\supp(D_h)$.
\end{enumerate}

\end{defn}

The interest in the neat models of pairs (that are equipped with fibrations), is two-fold. First, the conditions (\ref{neat}.\ref{neat1}) and (\ref{neat}.\ref{neat3}) ensure that $(\wtilde h)^*$ defines a well-defined pull-back map from symmetric $\mathcal{C}$-differential forms $\bigl(\Omega_{Z_h}\log(\Delta_h)\bigr)^{\otimes_{\mathcal{C}}N}$ attached to $(Z_h,\Delta_h)$ to the sheaf of tensorial  logarithmic forms $\bigl(\Omega_{Y_h}\log (D_h)\bigr)^{\otimes N}$ (see the discussion before the Definition~\ref{C-base}). Secondly, according to the property (\ref{neat}.\ref{neat2}), the neat model $(Y_h,D_h)$ inherits the birational properties of the original pair $(Y,D)$. For example if $(Y,D)$ special, then so is $(Y_h,D_h)$. These attributes will be crucial to the proof of the main result (Theorem.~\ref{reduction}) of this section.

\begin{prop}[Construction of Neat Models, cf. \protect{\cite[Sect.~10]{JK11b}}]\label{construction} Every normal logarithmic pair $(Y,D)$ and a surjective morphism with connected fibers $h:Y \to Z$, where $Z$ is a projective variety, admits a neat model.

\end{prop}

\begin{proof}Let $\alpha_1:Z_1 \to Z$ be a suitable modification of the base of the fibration $h$ such that the normalization of the induced fiber product $Y\times_Z Z_1$, which we denote by $Y_1$, givers rise to an equidimensional fibration $h_1:Y_1 \to Z_1$, i.e. a \emph{flattening} of $h$, and a birational map $\mu_1:Y_1 \to Y$ (see the diagram below). Define $D_1$ to be the maximal reduced divisor contained in the $\supp(\mu_1^{-1}D)$ and let $$D_1=D_1^{\text{ver}}+D_1^{\text{hor}}$$ be the decomposition of $D_1$ into sum of its vertical $D_1^{\text{ver}}$ and horizontal $D_1^{\text{hor}}$ components. Introduce a closed subset in $Z_1$ by $D_{Z_1}:=h_1(D_1^{\text{ver}})$. Let $\Delta_1\subset Z_1$ denote the log-discriminant locus defined by the fibration $h_1$ and the divisor $D_1$. Now, let $\alpha_2:Z_h \to Z_1$ be a desingularization of $Z_1$ such that the maximal reduced divisor in the $\supp(\alpha_2^{-1}\Delta_1 \cup \alpha_2^{-1}D_{Z_1})$ is snc. Set $Y_2$ to be the normalization of the fiber product $Y_1\times_{Z_1}Z_2$, and $\mu_2$ the naturally induced birational morphism. Define $D_2$ in $Y_2$ by the maximal reduced divisor contained in the $\supp(\mu_2^{-1}D_1)$. Finally let $\mu_3:Y_h \to Y_2$ be a log-resolution of $(Y_2,D_2)$ and take $\wtilde h:Y_h \to Z_h$ to be the induced fibration.

$$
\xymatrix{
    Y   \ar[d]^h && Y_1 \ar[ll]_{\mu_1} \ar[d]^{h_1} &&  Y_2 \ar[ll]_{\mu_2}  \ar[d]^{h_2} && Y_h \ar[ll]_{\mu_3} \ar[dll]^{\wtilde h}     \\
   Z  &&  Z_1 \ar[ll]_{\alpha_1} && Z_h \ar[ll]_{\alpha_2}
 }    
$$

\noindent Now set $\wtilde D_2$ to be the maximal reduced divisor in $\supp(\mu_3^{-1})$. Note that $h_1$ remains equidimensional under the base change of $\alpha_2$, i.e. $h_2$ is also equidimensional. This implies that when we desingularize $Y_2$ by $\mu_3$, every $\wtilde h$-exceptional divisor is $\mu_3$-exceptional. Let $E_3$ be the sum of all $\wtilde h$-exceptional prime divisors in $Y_h$ and define $D_h:=\wtilde D_2+E_3$ to be the extension of $\wtilde D_2$ by $E_3$. We finish by defining the birational morphisms $\mu$ and $\alpha$ in Definition~\ref{neat} by $(\mu_3 \circ \mu_2 \circ \mu_1)$ and $(\alpha_2\circ \alpha_1)$, respectively. Now by construction, the $\mathcal{C}$-structure $\Delta_h$ on $Z_h$ induced by $D_h$ and $\wtilde h$ defines a smooth pair $(Z_h,\Delta_h)$, as required.

\end{proof}

\begin{thm}[Reduction of the Isotriviality Conjecture]\label{reduction} The isotriviality conjecture~\ref{iso} holds, if the following assertion is true: 

\begin{enumerate} 
\item\label{assert} Let $(T,B)$ be a smooth pair. If $\bigl(\Omega_T\log (B)\bigr)^{\otimes_{\mathcal{C}}N}$ admits an invertible subsheaf $\sL$ with $\kappa_{\mathcal{C}}(T,\sL)=\dim T$, then $(T,B)$ is of log-general type.

\end{enumerate}

\end{thm}

\begin{proof} Let $f^\circ:X^{\circ} \to Y^{\circ}$ be a smooth family of canonically-polarized manifolds, where $Y^\circ$ is a special quasi-projective variety. We assume that $\dim(Y^\circ)>0$ (otherwise there is nothing to prove). Let $Y$ be a smooth compactification with boundary divisor $D$ such that $D \cong Y\backslash Y^{\circ}$ and that the induced map $\wtilde {\mu}: Y \to \overline{\mathfrak M}$ to a compactification of $\mathfrak{M}$ is a morphism. Aiming for a contradiction, assume that the family $f^\circ:X^\circ \to Y^\circ$ is \emph{not} isotrivial, that is $\dim(\mathrm{Im}(\wtilde \mu))>0$. Now, if $\wtilde {\mu}$ is generically finite, then thanks to Campana and P\u{a}un's solution to Viehweg's conjecture (Conjecture~\ref{VC}), we find that $(K_Y+D)$ is big, contradicting the assumption that $(Y,D)$ is special. Therefore to prove the theorem, we only need to treat the case where $\wtilde {\mu}:Y \to \overline{\mathfrak M}$ is \emph{not} generically finite (so that $\wtilde \mu$ has positive-dimensional general fibres). In this case, by the Stein factorization, we can find a projective variety $Z$ such that the morphism $\wtilde \mu$ factors through a fibration with connected fibres $h:Y \to Z$ and a finite morphism $Z \to \overline{\mathfrak M}$. According to Proposition~\ref{construction}, we can find a neat model $(Y_h,D_h)$ of the pair $(Y,D)$ and the fibration $h:Y\to Z$. 

$$
  \xymatrix{
    &&  Y  \ar[dll]_{\wtilde \mu}  \ar[d]_{h}  &Y_h  \ar[l]_{\mu} \ar[d]^{\wtilde h}  \\
     \overline{\mathfrak M}  && Z \ar[ll]_{\text{finite}}  & Z_h \ar[l]_{\alpha}
  }
  $$
We observe that since $Y_h\backslash D_h$ is isomorphic to an open subset of $Y^{\circ}$, it also parametrizes a smooth family of canonically polarized manifolds. Thus by~\cite[Thm.~1.4]{VZ02}, for some positive integer $N$, we can find a line subbundle $\sL\subseteq \bigl(\Omega_{Y_h}\log (D_h)\bigr)^{\otimes N}$ such that $\kappa(Y_h,\sL)\geq \dim Z_h$. Moreover by~\cite[Thm.~1.4]{JK11a}, we know that the Viehweg-Zuo subsheaf $\sL$ generically comes from the coarse moduli space. More precisely, there exists an inclusion $\sL\subseteq \sB^{\otimes N}$, where $\sB$ is the saturation of the image of 

$$d\wtilde h: (\wtilde h)^*(\Omega_{Z_h}) \to \Omega_{Y_h}\log (D_{h}).$$

Let us now collect the various properties of the pairs $(Y_h,D_h)$ and $(Z_h,\Delta_h)$, and the fibration $\wtilde h:Y_h\to Z_h$ (recall that, by definition, the divisor $\Delta_h$ is the $\mathcal{C}$-base of the fibration $\wtilde h:(Y_h,D_h)\to Z_h$), that we have found so far:

\begin{changemargin}{0.5cm}{0.5cm}
\refstepcounter{equation}\theequation. $(Y_h,D_h)$ and $(Z_h,\Delta_h)$ are both smooth pairs (property (\ref{neat}.\ref{neat1})).

\noindent  \refstepcounter{equation}\theequation. $D_{h}$ contains all $\wtilde h$-exceptional prime divisors (property (\ref{neat}.\ref{neat3})).

\noindent  \refstepcounter{equation}\theequation. There exists a saturated rank-one subsheaf $\sL \subseteq \sB^{\otimes N}$, for some positive integer $N$, such that $\kappa(Y_h,\sL)\geq \dim Z_h$.

\end{changemargin}
With these conditions, we can apply~\cite[Cor.~5.8]{JK11a} to find a saturated rank-one subsheaf $\sL_{Z_h}\subseteq \bigl(\Omega_{Z_h}\log (\Delta_h)\bigr)^{\otimes_{\mathcal{C}}N}$ such that 
\begin{equation}\label{sub}
     \kappa_{\mathcal{C}}(Z_h,\sL_{Z_h})= \kappa(Y_h,\sL)\geq \dim(Z_h).
        \end{equation}

Finally, if the statement (\ref{reduction}.\ref{assert}) holds, then $(Z_h,\Delta_h)$ is of log-general type. On the other hand by property (\ref{neat}.\ref{neat2}), for every $1\leq p \leq n$, we can push-forward invertible subsheaves of $\Omega^p_{Y_h}\log (D_{h})$ to those of $\Omega^p_Y\log (D)$. In particular, since $(Y,D)$ is special, then so is $(Y_h,D_{h})$. But, this is a contradiction to our previous finding that $(K_{Z_h}+\Delta_{h})$ is big (recall that for a neat model  $\wtilde h:Y_h\to Z_h$, and for sufficiently divisible positive integer $m$, we always have 
\begin{equation}
h^0(Y_h,\sG^{\otimes m})=h^0\bigl(Z_h,\sO_{Z_h}(K_{{Z_h}}+\Delta_h)^{\otimes m}\bigr),
\end{equation}
where $\sG$ denotes the saturation of the pull-back bundle $(\wtilde h)^* \bigl(\sO_{Z_h}(K_{Z_h})\bigr)$ inside $\Omega_{Y_h}^{\dim(Z_h)}\log (D_h)$).

\end{proof}

\

\section{The Isotriviality conjecture: The approach of Campana and P\u{a}un}\label{last}

In this section we prove the statement (\ref{reduction}.\ref{assert}) in the previous section. The isotriviality conjecture will then follow from Theorem~\ref{reduction}. The proof is completely based on the solution of~\cite[Sect.~4]{CP13} to the Viehweg's hyperbolicity conjecture~\ref{VC}. In particular, Theorem~\ref{final} should be taken as the generalization of~\cite[Thm.~4.1]{CP13} from the category of purely logarithmic smooth pairs (the boundary divisor is reduced) to that of smooth pairs in general.

For the ease of notation we have replaced the pair $(T,B)$ in the reduction statement (\ref{reduction}.\ref{assert}) by $(X,D)$ with the warning that $D$ should not be confused with the boundary divisor of the compactification of $Y^{\circ}$ that was introduced in the previous sections. 

\begin{prop}\label{cp} Let $(X,D)$ be a smooth pair of dimension $n$ and $\sL\subseteq \bigl(\Omega_X\log (D)\bigr)^{\otimes_{\mathcal{C}}N}$ a saturated rank one subsheaf with $\kappa_{\mathcal{C}}(X,\sL)=\dim X$. For every ample divisor $A$ in $X$, there exists a rational number $c=c(A,\sL)\in \bQ^+$, depending on $A$ and $\sL$, such that the inequality

\begin{equation}\label{volume}
   \vol(K_X+D+G)\geq c\cdot \vol (A),
     \end{equation}

\noindent holds for every $\mathbb{Q}$-Cartier divisor $G$ satisfying the following properties:

\begin{changemargin}{0.5cm}{0.5cm}
\refstepcounter{equation}\theequation.\label{prop1} $(D+G) \sim_{\mathbb{Q}} P$, for some big $\mathbb{Q}$-Cartier divisor $P$ such that $\lfloor P \rfloor =0$.

\noindent  \refstepcounter{equation}\theequation.\label{prop2} $(X,D+G)$ and $(X,P)$ are both smooth pairs.

\noindent  \refstepcounter{equation}\theequation. \label{prop3} The $\bQ$-Cartier divisor $(K_X+D+G)$ is pseudo-effective. 

\end{changemargin}

\end{prop}

\begin{proof} First, let us fix an ample divisor $A$. We notice that by an argument similar to that of Kodaira's lemma~\cite[Prop.~2.2.6]{Laz04}), we can always find a (sufficiently large) positive integer $m$ such that 
\begin{equation*}
H^0\bigl(X,(\sL)^{\otimes_{\mathcal{C}}N}\otimes \sO_X(-A)\bigr)\neq 0.
\end{equation*}
Let the invertible subsheaf $\sL'\subseteq \bigl(\Omega_X\log(D)\bigr)^{\otimes_{\mathcal{C}}(m.N)}$ denote the line-bundle $\sL^{\otimes_{\mathcal{C}}m}$, so that the inequality 
\begin{equation}\label{ineq:0}
A\leq L'
\end{equation}
holds between Cartier divisors $L'$ and $A$, $L'$ being the divisor verifying the isomorphism $\sO_X(L')\cong \sL'$. We shall prove the proposition in two steps. First, we run the log-minimal model program (or LMMP, for short) for the smooth pair $(X,P)$. We notice that since $P$ is big and has no reduced components (assumption (\ref{prop1})),  according to~\cite[Thm.~1.1]{BCHM10}, after a finite number of divisorial contractions and log-flips, the program terminates in a log-minimal model $(X',P')$, i.e. $(K_{X'}+P')$ is nef. Here, at the minimal level, we shall find a lower-bound for $\vol(K_{X'}+P')$ in terms of $\vol(A)$ and \emph{independent of $G$}. The second step of the proof is standard; we will just use the negativity lemma in the minimal model theory and replace $\vol(K_{X'}+P')$ by $\vol(K_X+P)$ to establish the required inequality (\ref{volume}).

\

\noindent \textbf{Step.~1: Log-minimal model of $(X,P)$ and the volume of its log-canonical divisor.} Let $\pi:(X,P)\dashrightarrow (X',P')$ be the birational map defined by the LMMP. Take $\mu:\wtilde X\to X$ to be a modification of $X$ resolving the indeterminacy of $\pi$, with resulting morphism $\wtilde \pi:\wtilde X \to X'$, and such that $\supp(\Exc(\mu)\cup \wtilde D\cup \wtilde G)$, where $\wtilde D$, $\wtilde G$ are the $\mu$-birational transforms of $D$ and $G$, respectively, is simple normal-crossing in $\wtilde X$:

$$
  \xymatrix{
 \wtilde Y \ar[rr]^{\gamma}_{\text{adapted cover}} && \wtilde X \ar[d]_{\mu}  \ar[drr]^(.4){\quad \wtilde \pi, \text{ birational}} \\
 && (X,P) \ar@{.>}[rr]^(.4){\pi}_(.4){LMMP} && (X',P')   
}
$$  

\noindent Let $\gamma:\wtilde Y\to \wtilde X$ be an adapted cover for the pair $(\wtilde X,\wtilde D+\wtilde G+E)$, where $E$ is the maximal reduced divisor contained in $\Exc(\mu)$. 
We notice that, as $\sL'$ is a subsheaf of $\bigl(\Omega_X\log(D)\bigr)^{\otimes_{\mathcal{C}}(m.N)}(\subseteq \bigl(\Omega_X\log(\ulcorner D\urcorner)\bigr)^{\otimes (m.N)})$, the inclusion  
\begin{equation*}
\mu^*(\sL')\subseteq \bigl(\Omega_{\wtilde X}\log(\wtilde D+\wtilde G+E)\bigr)^{\otimes_{\mathcal{C}}(m.N)}.
\end{equation*}
follows from the definition. Now, in order for us to use the generic semi-positivity result (Corollary.~\ref{invariance}), we need $(K_{\wtilde X}+\wtilde D+\wtilde G+E)$ to be pseudo-effective. This is indeed the case:  from the ramification formula for $\mu$ we have $(K_{\wtilde X}+\wtilde D+\wtilde G)=\mu^*(K_X+D+G)+\wtilde E$, $\wtilde E$ being an effective exceptional divisor (the effectivity follows from our assumption that $(X,D+G)$ is a smooth pair (\ref{prop1})). So, from the pseudo-effectivity of $(K_X+D+G)$ (assumption (\ref{prop3})) it follows that $(K_{\wtilde X}+\wtilde D+\wtilde G)$ is pseudo-effective, and thus so is $(K_{\wtilde X}+\wtilde D+\wtilde G+E)$, as required. Therefore Corollary~\ref{invariance} applies and the inequality 
\begin{equation*}
\mu^*(L')\cdot P^{n-1} \leq u(K_{\wtilde X}+\wtilde D+\wtilde G+E)\cdot P^{n-1}
\end{equation*}
holds, where $u:=(mN)(n^{mN})^{(mN-1)}$, for any nef divisor $P$ in $\wtilde X$. In particular, for any fixed ample divisor $H'$ in $X'$ and positive integer $r$, we have

\begin{IEEEeqnarray}{rCl}\label{ineq:1}
\mu^*(L')\cdot \wtilde \pi^*(K_{X'}+P'+\frac{1}{r}H')^{n-1} & \leq & u.(K_{\wtilde X}+\wtilde D+\wtilde G+E)\cdot \nonumber\\
&& \cdot \ \wtilde \pi^*(K_{X'}+P'+\frac{1}{r}H')^{n-1}.
\end{IEEEeqnarray}

Now, let $U$ be a Zariski open subset of $X'$ of $\codim_{X'}(X'\backslash U)\geq 2$ where $\pi^{-1}|_{U}$ and $\wtilde \pi^{-1}|_{U}$ are both isomorphisms. For every $r\in \bN^+$, define $d_r$ to be a sufficiently large positive integer such that the linear system $|d_r(K_{X'}+P'+\frac{1}{r}H')|$ is basepoint-free and that the irreducible curve $C_r:=B_r^1\cap \ldots \cap B_r^{n-1}$, cut out by general members $B_r^i\in |d_r(K_{X'}+P'+\frac{1}{r}H')|$, is a subset of $U$. We notice that as $C_r\subset U$, and because of our assumption (\ref{prop1}), the right-hand side of the inequality (\ref{ineq:1}) is equal to  $(\frac{1}{d_r})^{n-1}u.(K_{X'}+P')\cdot (K_{X'}+P'+\frac{1}{r})^{n-1}$. Therefore, we may write the inequality (\ref{ineq:1}) as 

\begin{equation*} (d_r)^{n-1}\mu^*(L')\cdot \wtilde \pi^*(K_{X'}+P'+\frac{1}{r}H')^{n-1}\leq u. (K_{X'}+P'+\frac{1}{r}H')^n,
\end{equation*}
so that 

\begin{equation}\label{ineq:2}
\mu^*(L')\cdot \wtilde \pi^*(K_{X'}+P'+\frac{1}{r}H')^{n-1}\leq u.\vol (K_{X'}+P'+\frac{1}{r}H').
\end{equation}

Next, we notice that, as $(L'-A)\geq 0$ (inequality (\ref{ineq:0})), the pull-back $\mu^*(L'-A)$ is also effective. Therefore, and again by using the fact that the nef cone in the N\'eron-Severi space $\mathrm{N}^1(\wtilde X)_{\bR}$ is equal to the closure of the ample one, we have $\mu^*(L'-A)\cdot \wtilde \pi^*(K_{X'}+P'+\frac{1}{r}H')^{n-1}\geq 0$. Hence we can rewrite the inequality (\ref{ineq:2}) as 

\begin{equation}\label{ineq:3}
\mu^*(A)\cdot \wtilde \pi^*(K_{X'}+P'+\frac{1}{r}H')^{n-1}\leq u.\vol (K_{X'}+P'+\frac{1}{r}H')
\end{equation}
Now, by applying the Teissier's inequality~\cite[Thm.~1.6.1]{Laz04} (to the left-hand side of the inequality (\ref{ineq:3})), we have

\begin{equation*}
\vol(A)^{\frac{1}{n}} \cdot \vol(K_{X'}+P'+\frac{1}{r}H')^{\frac{n-1}{n}}\leq u.\vol(K_{X'}+P'+\frac{1}{r}H'), 
\end{equation*}
i.e. 

\begin{equation}\label{ineq:4}
\vol(A)^{\frac{1}{n}}\leq u.\vol (K_{X'}+P'+\frac{1}{r}H')^{\frac{1}{n}}.
\end{equation}
Finally, thanks to the continuity of $\vol(.)$, by taking $r\to \infty$ in the inequality (\ref{ineq:4}) we have 
\begin{equation}\label{ineq:5}
\frac{1}{u^n} \cdot \vol (A)\leq \vol (K_{X'}+P'),
\end{equation}
that is, the inequality (\ref{volume}) holds for the log minimal model $(X',P')$, if we take $c:=1/u^n$.

\
    
\noindent \textbf{Step.~2: Lower-bound for the volume of $(K_X+P)$.} By the negativity lemma in the minimal model theory, we know that $H^0(X,m(K_X+P))\cong H^0(X',m(K_{X'}+P'))$, for all $m\in \mathbb{N}^+$.  In particular the equality $\vol(X,K_X+P)=\vol(X',K_{X'}+P')$ holds. The required inequality (\ref{volume}) now follows form the inequality (\ref{ineq:5}) in the previous step and assumption (\ref{prop1}).

\end{proof}

\

\begin{thm}\label{final} Let $(X,D)$ be a smooth pair and $\sL\subseteq \bigl(\Omega_X\log (D)\bigr)^{\otimes_{\mathcal{C}}N}$ an invertible subsheaf. If $\kappa_{\mathcal{C}}(X,\sL)=\dim X$  , then $(K_X+D)$ is big.

\end{thm}

\begin{proof} Let $H$ be a very ample divisor such that $H-D$ is ample, and let $r$ be a (fixed) sufficiently large positive integer for which the divisor $\bigl(r(H-D)\bigr)$ is very ample. Define the  hyperplane section $B_{D}$ to be a general member of the linear system $|r(H-D)|$. From construction it follows that, for every integer $M>r$, the $\mathbb{Q}$-divisor $(D+\frac{1}{M}B_{D})$ is $\mathbb{Q}$-linearly equivalent to an snc divisor, which we denote by $P_M$, with no reduced components:

\begin{align*}
D+\frac{1}{M}B_{D} &\sim_{\mathbb{Q}}D+\frac{1}{M}\bigl(r(H-D)\bigr)\\
     &=(1-\frac{r}{M})D+\frac{r}{M}H=:P_M.\\
  \end{align*}

\begin{subclaim}\label{claim} The divisor $(K_X+P_M)$ is pseudo-effective, for all integers $M$ verifying the inequality $M>r$.

\end{subclaim}
\noindent Let us for the moment assume that the claim holds. Define the $\mathbb{Q}$-Cartier divisor $G$ in Proposition~\ref{cp} by $G:=\frac{1}{M}B_{D}$. As the conditions (\ref{prop1}), (\ref{prop2}) and (\ref{prop3}) in Proposition~\ref{cp} are all satisfied, it follows from the inequality (\ref{volume}) that for any fixed ample divisor $A$, there exists a constant $c$ such that 

\begin{equation}\label{eq8}
   \vol(K_X+D+\frac{1}{M}B_{D}) \geq c\cdot \vol(A),\ \   \forall M\in \mathbb{N} \text{\ such that \  } M>r.
       \end{equation}
Therefore, by taking $M\to \infty$, the continuity property of $\vol(.)$ and the fact that \emph{the constant $c$ in Proposition~\ref{cp} is independent of $M$}, it follows that the divisor $(K_X+D)$ is big.

It now remains to prove the claim~\ref{claim}.

\

\noindent \emph{Proof of claim~\ref{claim}}. Aiming to extract a contradiction, suppose that $(K_X+P_M)$ is not pseudo-effective for some positive integer $M>r$. Let $H'$ be a suitably-chosen very ample divisor such that the effective log-threshold given by 

$$\epsilon:=min \{ t\in \mathbb{R}^+: K_X+P_M+tH' \textrm{\;\ is pseudo-effective}\},$$ is smaller than $1$. According to~\cite[Cor. 1.1.7]{BCHM10} $\epsilon$ is rational. Now by applying Proposiotion~\ref{cp} to the pair $(X,D)$ with $G:=\frac{1}{M}B_{D}+\epsilon H'$, we find that $K_X+P_M+\epsilon H'$ is big. But as the big cone forms the interior of the cone of pseudo-effective $\mathbb{Q}$-Cartier classes, for sufficiently small $\delta$, $K_X+D_M+(\epsilon-\delta)H'$ is also pseudo-effective, contradicting the minimality assumption on $\epsilon$.

\end{proof}

The isotriviality conjecture (Conjecture~\ref{iso}) now follows from Theorem~\ref{final} together with  Theorem~\ref{reduction} in the previous section.

\

\providecommand{\bysame}{\leavevmode\hbox to3em{\hrulefill}\thinspace}
\providecommand{\MR}{\relax\ifhmode\unskip\space\fi MR }
\providecommand{\MRhref}[2]{%
  \href{http://www.ams.org/mathscinet-getitem?mr=#1}{#2}
}
\providecommand{\href}[2]{#2}


\begin{thebibliography}{BCHM10}


\bibitem[BCHM10]{BCHM10}
Caucher Birkar, Paolo Cascini, Christopher D.Hacon, and James McKernan, \emph{Existence of minimal models for varieties of log general type}, Journal of the AMS 23 (2010), 405-468.


\bibitem[BDPP04]{BDPP}
S{\'e}bastien Boucksom, Jean-Pierre Demailly, Mihai P\u{a}un, and Thomas Peternell, \emph{The pseudo-effective cone of a compact {K}\"ahler manifold and varieties of negative Kodaira dimension}, J. Algebraic. Geom. \textbf{22} (2013), 201-248. doi:10.1090/S1056-3911-2012-00574-8.	
     
     
  
\bibitem[Cam04]{Ca04} 
 Fr{\'e}d{\'e}ric Campana, \emph{Orbifolds, special varieties and classification theory, Ann. Inst. Fourier (Grenoble)}, Vol54, no. 3, 499-630 (2004)


\bibitem[Cam08]{Ca08}
 Fr{\'e}d{\'e}ric Campana, \emph{Orbifoldes sp\'ciales et classification bim\'eromorphe des vari\'et\'es {K}\"ahl\'eriennes compactes}, arXiv:0705.0737v5 (2008).



\bibitem[CP13]{CP13}
Fr{\'e}d{\'e}ric Campana, and Mihai P\u{a}un, \emph{Orbifold generic semi-positivity: an application to families of canonically polarized families}, version~1, arXiv:1303.3169.



\bibitem[Iit82]{Iit82} Shigeru iitaka, \emph{Algebraic geometry}, Springer-Verlag, New York, 1977, Graduate texts in Mathematics, No. 52. 57. Number 3116.


\bibitem[JK11a]{JK11a}
Kelly Jabbusch, and Stefan Kebekus, \emph{Positive sheaves of differentials coming from coarse moduli spaces}, Annales de l'institut Fourier, 61 no. 6 (2011), p. 2277-2290.


\bibitem[JK11b]{JK11b}
\bysame, \emph{Families over special base manifolds and a conjecture of Campana}, Mathematische Zeitschrift: Volume 269, Issue 3 (2011), Page 847-878.



\bibitem[KK08]{KK08}
Stefan Kebekus, and S\'andor Kov\'acs, \emph{The structure of surfaces and threefolds mapping to the moduli stack of canonically polarized varieties}, Duke Math. J. vol.~155, Number 1, pp.~1-33, 2010.


\bibitem[Laz04]{Laz04}
Robert Lazarsfeld, \emph{Positivity in Algebraic Geometry I} , Springer, Vol. 48.


\bibitem[Lu02]{Lu02}
Steven S. Y. Lu, \emph{A refined Kodaira dimension and its canonical fibration}, arXiv:math/0211029.


\bibitem[Miy87a]{Miy87}
Yoichi Miyaoka, \emph{The {C}hern classes and {K}odaira dimension of a minimal
  variety}, Algebraic geometry, Sendai, 1985, Adv. Stud. Pure Math., vol.~10,
  North-Holland, Amsterdam, 1987, pp.~449--476. \MR{89k:14022}


\bibitem[Miy87b]{Miy85}
\bysame, \emph{Deformations of a morphism along a foliation and applications},
  Algebraic geometry, Bowdoin, 1985 (Brunswick, Maine, 1985), Proc. Sympos.
  Pure Math., vol.~46, Amer. Math. Soc., Providence, RI, 1987, pp.~245--268.
  \MR{MR927960 (89e:14011)}


\bibitem[Vie83]{Vie83}
Eckart Viehweg, \emph{Weal positivity and the additivity of the Kodaira dimension for certain fibre spaces}, Algebraic varieties and analytic varieties (Tokyo, 1981), Adv. Stud. Pure Math., Vol. 1, North-Holland, Amsterdam, 1983, pp. 329-353.


\bibitem[Vie95]{Vie95}
Eckart Viehweg, \emph{Quasi-projective moduli for polarized manifolds}, Ergebnisse der Mathematik und ihrer Grenzgebiete (3) [Results in Mathematics and Related Areas (3)], vol.~30, Springer-Verlag, Berlin, 1995.


\bibitem[VZ02]{VZ02}
Eckart Viehweg and Kang Zuo, \emph{Base spaces of non-isotrivial families of smooth minimal models}, Complex geometry (G\"ottingen, 2000), Springer, Berlin, 2002, pp. 279-328.



\bibitem[VZ03]{VZ03}
Eckart Viehweg, \emph{On the Brody hyperbolicity of moduli spaces for canonically polarized manifolds}, Duke Math. J.~118 (2003), no.~1, 103-150.



\end{thebibliography}
\end{document}